\documentclass[runningheads]{llncs}
\usepackage[utf8]{inputenc}
\usepackage{lipsum}
\usepackage{amsfonts}
\usepackage{amsmath}
\usepackage{graphicx}
\usepackage{epstopdf}
\usepackage{enumitem}
\usepackage{algorithmic}
\usepackage{float}
\usepackage[caption = false]{subfig}
\usepackage{graphicx}

\spnewtheorem{assumption}[theorem]{Assumption}{\bfseries}{\itshape}
\spnewtheorem{algorithm}[theorem]{Algorithm}{\bfseries}{\itshape}

\begin{document}

\title{Convergence Properties of a \\ Randomized Primal-Dual Algorithm \\ with Applications to Parallel MRI\thanks{MJE and CD acknowledge support from the EPSRC (EP/S026045/1). MJE is also supported by EPSRC (EP/T026693/1), the Faraday Institution
(EP/T007745/1) and the Leverhulme Trust (ECF-2019-478). EBG acknowledges the Mexican Council of Science and Technology (CONACyT).}}
\titlerunning{Convergence of a Random Primal-Dual Algorithm}
\author{Eric B. Guti\'errez\inst{1}
\and Claire Delplancke \inst{1}
\and Matthias J. Ehrhardt \inst{1,2}}
\authorrunning{E.B. Guti\'errez, C. Delplancke, and M.J. Ehrhardt}
\institute{Department of Mathematical Sciences, University of Bath, Bath, UK  \and Institute for Mathematical Innovation, University of Bath, Bath, UK\\ \email{ebgc20@bath.ac.uk, cd902@bath.ac.uk, M.Ehrhardt@bath.ac.uk}}
\maketitle

\begin{abstract}
  The Stochastic Primal-Dual Hybrid Gradient (SPDHG) was proposed by Chambolle et al. (2018) and is an efficient algorithm to solve some nonsmooth large-scale optimization problems. In this paper we prove its almost sure convergence for convex but not necessarily strongly convex functionals. We also look into its application to parallel Magnetic Resonance Imaging reconstruction in order to test performance of SPDHG. Our numerical results show that for a range of settings SPDHG converges significantly faster than its deterministic counterpart.
\end{abstract}

\section{Introduction}

Optimization problems have numerous applications among many fields such as imaging, data science or machine learning, to name a few. Many such problems in these areas are often formulated as
\begin{equation} \label{min}
    \hat{x} \in \arg\min_{x\in X} \sum_{i=1}^n f_i(A_ix) + g(x) 
\end{equation}
where $f_i:Y_i\to\mathbb{R}\cup\{\infty\}$ and $g:X\to\mathbb{R}\cup\{\infty\}$ are convex functionals, and $A_i:X\to Y_i$ are linear operators between finite-dimensional real Hilbert spaces.

Examples of such problems are total variation regularized image reconstruction \cite{fercoqBianchi2019,ROF} such as image denoising \cite{chambolle2016introduction} or PET reconstruction \cite{ehrhardt2019faster}; regularized empirical risk minimization \cite{shalevZhang,zhangXiao} such as support vector machine (SVM) \cite{boserGuyonVapnik} or least  absolute  shrinkage  and  selection operator (LASSO) \cite{cevherBigData}; and optimization with large number of constraints \cite{fercoq_etal2019,patrascuNecoara2017}, among others.

While some classical approaches such as gradient descent are not applicable when the functionals $f_i$ or $g$ are not smooth \cite{chambolle2016introduction}, primal-dual methods are able to find solutions to (\ref{min}) without assuming differentiability. For convex, proper and lower-semicontinuous functionals $f_i,g$, a primal-dual formulation for (\ref{min}) reads
\begin{equation} \label{saddle}
    \hat{x},\hat{y} \in \arg\min_{x\in X}\max_{y\in Y} \sum_{i=1}^n \langle A_ix,y_i \rangle - f_i^*(y_i) + g(x)
\end{equation}
where $f^*$ is the \textit{convex conjugate} of $f$ and $Y=\Pi_{i=1}^n Y_i$. We refer to any solution $\hat{w}=(\hat{x},\hat{y})$ of (\ref{saddle}) as a \emph{saddle point}.

A well-known example of primal-dual methods that solve (\ref{saddle}) is the Primal-Dual Hybrid Gradient (PDHG) \cite{chambollePock,esserZhangChan,CremersChambollePock}, as presented by Chambolle \& Pock (2011).
It naturally breaks down the complexity of (\ref{min}) into separate optimization problems by doing separate updates for the primal and dual variables $x,y$, as shown in (\ref{det}).
PDHG is proven to converge to a solution of (\ref{saddle}), however its iterations become very costly for large-scale problems, e.g. when $n\gg1$ \cite{spdhg}.

More recently, Chambolle et al. proposed the Stochastic Primal-Dual Hybrid Gradient (SPDHG) \cite{spdhg} which reduces the per-iteration computational cost of PDHG by randomly sampling the dual variable: at each step, instead of the full dual variable $y$, only a random subset of its coordinates $y_i$ gets updated. This offers significantly better performance than the deterministic PDHG for large-scale problems \cite{spdhg}. Examples of similar random primal-dual algorithms are found in 
\cite{fercoqBianchi2019,gao2019randomPD,latafat2019randomPD,zhangXiao}.

We are interested in the convergence of SPDHG. In \cite{spdhg}, it is shown that, for arbitrary convex functionals $f_i$ and $g$, SPDHG converges in the sense of \textit{Bregman distances}, which does not imply convergence in the norm. 
In this paper we present a proof for the almost sure convergence of SPDHG for convex but not necessarily strongly convex functionals, using alternative arguments to the recently proposed proof by Alacaoglu et al. (\cite{alacaoglu}, Theorem 4.4). In contrast to~\cite{alacaoglu}, where they represent SPDHG by using projections of a single iterative operator, we present a more intuitive representation through a random sequence of operators (Lemma~\ref{fixpoints}). 
A summary of the proof is laid out in Section~\ref{sec:ske} and the complete proof is detailed in Section~\ref{sec:conv}. A comparison with the proof of Alacaouglu et al. and other related work is discussed in Section~\ref{sec:rel}.

Finally, in Section~\ref{sec:numeric} we look into the application of parallel Magnetic Resonance Imaging (MRI) in order to compare the performance of SPDHG with that of the deterministic PDHG.

\section{Algorithm} In order to solve (\ref{saddle}), the deterministic PDHG with dual extrapolation~\cite{chambollePock} reads
\begin{equation} \label{det}
\begin{aligned}
    x^{k+1} & =\mathrm{prox}_{\tau g}(x^{k} - \tau A^T\bar{y}^{k}) \\
    y^{k+1} &  = 
    \mathrm{prox}_{\sigma f^{*}}(y^{k}+\sigma Ax^{k+1})
\end{aligned}
\end{equation}
\vspace{0mm}
where $\bar{y}^k = 2y^k - y^{k-1}$ is an extrapolation on the previous iterates, $A$ and $f$ are given by $Ax=(A_1x,...,A_nx)$ and $f(y)=\sum_{i}f_i(y_i)$, and the \textit{proximity operator} of any functional $f$ is defined as $\mathrm{prox}_{\sigma f}(v) := \arg\min_{y\in Y}\frac{\|v-y\|^2}{2} + \sigma f(y)$. 

SPDHG, in contrast, reduces the cost of iterations by only partially updating the dual variable $y=(y_i)_{i=1}^n$: at every iteration $k$, choose $j\in\{1,...,n\}$ at random with probability $p_i=\mathbb{P}(j=i)>0$, so that only the variable $y_{j}^{k+1}$ is updated, while the rest remain unchanged, i.e. $y_i^{k+1} = y_i^{k}$ for $i\neq j$. 

\begin{algorithm}[SPDHG]
\label{alg_spdhg}
\begin{algorithmic}
\STATE{Choose $\tau,\sigma_i>0$ and $x^0\in X$. Set $y^0=\textbf{0}\in Y$ and $z^0=\bar{z}^0=\textbf{0}\in X.$ }
\STATE{For $k\geq0$ do}
\STATE{ \vspace{-6mm}
\begin{equation*}
\begin{aligned}
& \hspace{11mm} \text{select } j^{k} \in\{1,...\,,n\} \text{ at random} \hspace{64mm} \\
& \begin{aligned} 
    \hspace{16mm} x^{k+1} & =\mathrm{prox}_{\tau g}(x^{k} - \tau\bar{z}^k) \\
    y_{i}^{k+1} &  = \left\{ 
    \begin{tabular}{ll}
    $\mathrm{prox}_{\sigma_i f_{i}^{*}}(y_{i}^{k}+\sigma_{i} A_{i}x^{k+1})$ \quad\quad &  if $i=j^{k}$ \\
    $y_{i}^{k}$ & else
    \end{tabular} \right. \\
    \delta^k  &= A^T_{j^{k}}(y_{j^{k}}^{k+1}-y_{j^{k}}^k) \\
           z^{k+1} &= z^k + \delta^k  \\
     \bar{z}^{k+1} &= z^{k+1} + {p^{-1}_{j^k}}\delta^k 
\end{aligned}
\end{aligned}
\end{equation*} }
\end{algorithmic}
\end{algorithm}


\section{Main Result}
\label{sec:main}
We establish the almost sure convergence of SPDHG for any convex functionals, under the same step size conditions as in \cite{spdhg}:

\begin{assumption} \label{assu}
We assume the following to hold:
\begin{enumerate}
    \item The set of solutions to (\ref{saddle}) is nonempty. 
    \item The functionals $g,f_i$ are convex, proper and lower-semicontinuous.
    \item The step sizes $\tau,\sigma_i>0$ satisfy
    \begin{equation} \label{assu_i}
        \tau\sigma_i\|A_i\|^2 < p_i \quad\text{for every } i.
    \end{equation} 
\end{enumerate}
\end{assumption}

\begin{theorem}[\textbf{Convergence of SPDHG}]
\label{main}
Let $(w^k)_{k\in\mathbb{N}}=(x^k,y^k)_{k\in\mathbb{N}}$ be a random sequence generated by Algorithm \ref{alg_spdhg}. Under Assumption~\ref{assu}, 
the sequence $(w^k)_{k\in\mathbb{N}}$ converges almost surely to a solution of (\ref{saddle}). 
\end{theorem}

\section{Sketch of the Proof} \label{sec:ske}

The following results lay out the proof of Theorem~\ref{main}. The complete proof is detailed in Section \ref{sec:conv}.
We use the notation $\|x\|^2_T=\langle Tx,x\rangle$, as well as the block diagonal operators $Q,S:Y\to Y$ given by $Q=\mathrm{diag}(p_1^{-1},...,p_n^{-1})$ and $S=\mathrm{diag}(\sigma_1,...,\sigma_n)$. The conditional expectation at time $k+1$ is denoted, for any functional $\varphi$, by
$ \mathbb{E}^{k+1}(\varphi(w^{k+1})) = \mathbb{E}(\varphi(w^{k+1}) | w^k) .$

The proof of Theorem \ref{main} uses the following important inequality from SPDHG, which is a consequence of (\cite{spdhg}, Lemma~4.4). This inequality is best summarized in (\cite{alacaoglu}, Lemma~4.1), which we have further simplified by using the fact that Bregman distances of convex functionals are nonnegative (\cite{spdhg}, Section~4).

\begin{lemma}[\cite{alacaoglu}, Lemma 4.1]
\label{ala}
Let $(w^k)_{k\in\mathbb{N}}$ be a random sequence generated by Algorithm~\ref{alg_spdhg} under Assumption~\ref{assu}. Then for every saddle point $\hat{w}$, 
\begin{equation} \label{inequality0}
\begin{aligned}
    V^k(w^k-\hat{w}) \geq \mathbb{E}^{k+1}(V^{k+1}(w^{k+1}-\hat{w})) + V(x^{k+1}-x^k,y^k-y^{k-1}) 
\end{aligned}
\end{equation}
where 
$V$ and $V^k$ are given by 
$V(x,y) = \|x\|^{2}_{\tau^{-1}} + 2\langle QAx,y \rangle + \|y\|^{2}_{QS^{-1}}$ and
$$ V^k(x,y) = \|x\|^{2}_{\tau^{-1}} - 2\langle QAx,y^k-y^{k-1} \rangle + \|y^k-y^{k-1}\|^2_{QS^{-1}} + \|y\|^{2}_{QS^{-1}}. $$
\end{lemma}

The following result is the central argument of our proof. It makes use of inequality~(\ref{inequality0}) and a classical result from Robbins \& Siegmund (Lemma~\ref{lemma2.2}) to establish an important convergence result. Its proof is detailed in Section~\ref{sec:conv}.

\begin{proposition} \label{tozero1}
Let $(w^k)_{k\in\mathbb{N}}$ be a random sequence generated by Algorithm \ref{alg_spdhg} under Assumption \ref{assu} and let $\hat{w}$ be a saddle point. Then:
\begin{enumerate}[label=\roman*)]
    \item \label{i} The sequence $(w^k)_{k\in\mathbb{N}}$ is a.s.\@ bounded.
    \item \label{ii} The sequence $(V^k(w^k-\hat{w}))_{k\in\mathbb{N}}$ converges a.s.
    \item \label{iii} The sequence $(\|w^k-\hat{w}\|)_{k\in\mathbb{N}}$ converges a.s.
    \item \label{iv} If every cluster point of $(w^k)_{k\in\mathbb{N}}$ is a.s.\@ a saddle point, the sequence $(w^k)_{k\in\mathbb{N}}$ converges a.s.\@ to a saddle point.
\end{enumerate}
\end{proposition}

Lastly, we prove that every cluster point of $(w^k)_{k\in\mathbb{N}}$ is almost surely a solution to (\ref{saddle}). This is also explained in Section~\ref{sec:conv}.

\begin{proposition} \label{cluster_saddle}
Let $(w^k)_{k\in\mathbb{N}}$ be a random sequence generated by Algorithm~\ref{alg_spdhg} under Assumption~\ref{assu}. Then every cluster point of $(w^k)_{k\in\mathbb{N}}$ is almost surely a saddle point.
\end{proposition}

\begin{proof}[Proof of Theorem \ref{main}]
By Proposition \ref{cluster_saddle}, every cluster point of $(w^k)_{k\in\mathbb{N}}$ is almost surely a saddle point and, by Proposition~\ref{tozero1} \ref{iv}, the sequence $(w^k)_{k\in\mathbb{N}}$ converges almost surely to a saddle point.
\end{proof}

\section{Proof of Convergence} \label{sec:conv}

This section contains detailed proofs for our two main arguments, Propositions~\ref{tozero1} and \ref{cluster_saddle}. The proof of Proposition~\ref{tozero1} follows a similar strategy to that of Combettes \& Pesquet in (\cite{combettesPesquet}, Proposition 2.3), and we have divided it into three sections. 

\subsection{Proof of Proposition \ref{tozero1} \ref{i}}
To show this first part, we borrow the following lemma from \cite{spdhg} (Lemma~4.2 with $c=1$, $v_i = \tau\sigma_i\|A_i\|^2$ and $\gamma^2=\max_i\frac{v_i}{p_i}$):

\begin{lemma}[\cite{spdhg}, Lemma 4.2]
\label{lemma4.2}
Let $p_i^{-1}\tau\sigma_i\|A_i\|^2 \leq \gamma^2<1$ for every $i$ and let $y^k$ be defined as in Algorithm \ref{alg_spdhg}. Then for every $x\in X$,
\begin{equation*}
    \mathbb{E}^k(V(x,y^k-y^{k-1})) \geq (1-\gamma)\mathbb{E}^k \big(\|x\|^{2}_{\tau^{-1}} + \|y^k-y^{k-1}\|^{2}_{QS^{-1}} \big) .  
\end{equation*}
\end{lemma}


\begin{proof}[Proof of Proposition \ref{tozero1} \ref{i}]
By (\cite{alacaoglu}, Lemma~4.1), for any
saddle point $\hat{w}$ we have
\begin{equation} \label{inequality}
    \Delta^k \geq \mathbb{E}^{k+1}(\Delta^{k+1}) + V(x^{k+1}-x^k,y^k-y^{k-1}) 
\end{equation}
where $\Delta^k = V^k(w^k-\hat{w})$.
By Lemma \ref{lemma4.2}, 
\begin{equation*}
    \mathbb{E}^k(V(x^{k+1}-x^k,y^k-y^{k-1})) \geq  (1-\gamma)\mathbb{E}^k \Big\{\|x^{k+1}-x^k\|^{2}_{\tau^{-1}} + \|y^k-y^{k-1}\|^{2}_{QS^{-1}} \Big\}.  
\end{equation*}
Hence, taking the full expectation in (\ref{inequality}) yields
\begin{equation*} 
    \mathbb{E}(\Delta^k) \geq \mathbb{E}(\Delta^{k+1}) 
    + (1-\gamma)\mathbb{E}\big(\|x^{k+1}-x^{k}\|^{2}_{\tau^{-1}} + \|y^{k}-y^{k-1}\|^{2}_{QS^{-1}}\big) .
\end{equation*}
Taking the sum from $k=0$ to $k=N-1$ gives 
\begin{equation} \label{sum}
    \Delta^0 \geq \mathbb{E}(\Delta^{N}) 
    + (1-\gamma)\mathbb{E}\Big\{\sum_{k=0}^{N-1} \|x^{k+1}-x^{k}\|^{2}_{\tau^{-1}} + \|y^{k}-y^{k-1}\|^{2}_{QS^{-1}}\Big\} 
\end{equation}
where $y^{-1}=y^0$. This implies $\Delta^0 \geq \mathbb{E}(\Delta^{N})$ and, by Lemma~\ref{lemma4.2} we have 
\begin{equation}\label{E(Delta^N)}
    \mathbb{E}(\Delta^N) \geq \mathbb{E}\big\{(1-\gamma)(\|x^N-\hat{x}\|^2_{\tau^{-1}} + \|y^N-y^{N-1}\|^2_{QS^{-1}}) + \|y^N-\hat{y}\|^2_{QS^{-1}}\big\}.
\end{equation}
It follows that $ \Delta^0 \geq (1-\gamma)\|x^N-\hat{x}\|^2_{\tau^{-1}} + \|y^N-\hat{y}\|^2_{QS^{-1}} $ a.s., from where it is clear that the sequence $(w^N)_{N\in\mathbb{N}}$ is bounded almost surely.
\end{proof}

\subsection{Proof of Proposition \ref{tozero1} \ref{ii}-\ref{iii}}
As in (\cite{combettesPesquet}, Proposition 2.3), we use a classical result from Robbins \& Siegmund:

\begin{lemma}[\cite{robbins1971}, Theorem 1] 
\label{lemma2.2}
Let $\mathcal{F}_k$ be a sequence of sub-$\sigma$-algebras such that $\mathcal{F}_k\subset \mathcal{F}_{k+1}$ for every $k$, and let $\alpha_k$, $\eta_k$ be nonnegative $\mathcal{F}_k$-measurable random variables such that $\sum_{k=1}^\infty \eta_k <\infty$ almost surely and
$$ \mathbb{E}(\alpha_{k+1}\,|\,\mathcal{F}_k) \,\leq\, \alpha_k + \eta_k \;\text{ a.s.}$$
for every $k$. Then $\alpha_k$ converges almost surely to a random variable in $[0,\infty)$.
\end{lemma}

\begin{proof}[Proof of Proposition \ref{tozero1} \ref{ii}-\ref{iii}]
From (\ref{E(Delta^N)}) we have $\mathbb{E}(\Delta^N)\geq0$. Thus taking the limit as $N\to\infty$ in (\ref{sum}) yields
\begin{equation} \label{ntozero}
     \mathbb{E}\Big\{\sum_{k=0}^\infty \|x^{k+1}-x^{k}\|^{2}_{\tau^{-1}} + \|y^{k}-y^{k-1}\|^{2}_{QS^{-1}}\Big\} < \infty
\end{equation}
which implies
\begin{equation}\label{sum(y^n)} 
    \sum_{k=0}^\infty \|x^{k+1}-x^{k}\|^{2}_{\tau^{-1}} + \|y^{k}-y^{k-1}\|^{2}_{QS^{-1}} < \infty \quad\text{a.s.}
\end{equation}
and, in particular,
\begin{equation}\label{yto0}
   \|y^k-y^{k-1}\|_{QS^{-1}}\to0 \quad\text{a.s.} 
\end{equation}
Since $(w^k)_{k\in\mathbb{N}}$ is bounded a.s., so is $(x^k)_{k\in\mathbb{N}}$ and, since the operators $Q$, $A$ and $S$ are also bounded, there exists $M>0$ such that, for every $k$,
\begin{equation*} \label{Q}
    |\langle QA(x^k-\hat{x}),y^{k}-y^{k-1} \rangle| \leq \|QA\|\|x^k-\hat{x}\|\|y^k-y^{k-1}\| \leq M\|y^k-y^{k-1}\|_{QS^{-1}} 
\end{equation*}
a.s. and therefore, by (\ref{yto0}),
\begin{equation} \label{qto0}
    \langle QA(x^k-\hat{x}),y^{k}-y^{k-1} \rangle \to 0 \quad \text{a.s.}
\end{equation}
The fact that $(w^k)_{k\in\mathbb{N}}$ is a.s.\@ bounded, together with (\ref{yto0}) and (\ref{qto0}) imply the sequence $(\Delta^k)_{k\in\mathbb{N}}$ is also a.s.\@ bounded. Thus there exists $\hat{M}\geq0$ such that $ \Delta^k+\hat{M} \geq 0 $ for every $k$. Let $\alpha_k = \Delta^k+\hat{M}$ and $\eta_k=2|\langle QA(x^{k+1}-x^{k}),y^{k}-y^{k-1} \rangle|$. From (\ref{inequality}) we deduce
\begin{equation} \label{supermartingale}
    \alpha_k + \eta_k \,\geq\, \mathbb{E}^{k+1}(\alpha_{k+1}) \quad\text{a.s. for every } k,
\end{equation}
where all the terms are nonnegative and,  for some $\tilde{M}>0$,
\begin{equation*}
\begin{aligned}
    \eta_k = 2|\langle QA(x^{k+1}-x^{k}),y^{k}-y^{k-1} \rangle| 
    &\leq 2\|Q A\| \|x^{k+1}-x^k\| \|y^k-y^{k-1}\| \\
    &\leq 2\tilde{M}  \|x^{k+1}-x^k\|_{\tau^{-1}} \|y^k-y^{k-1}\|_{QS^{-1}} \\
    &\leq \tilde{M} \big( \|x^{k+1}-x^k\|_{\tau^{-1}}^2 + \|y^k-y^{k-1}\|^2_{QS^{-1}} \big) 
\end{aligned}
\end{equation*}
which implies, by (\ref{sum(y^n)}), $\sum_{k=1}^\infty\eta_k<\infty$ a.s.. Thus (\ref{supermartingale}) satisfies all the assumptions of Lemma~\ref{lemma2.2} and it yields $\Delta^k \to \alpha$ a.s.
for some $\alpha\in[-\hat{M},\infty)$. Furthermore, from (\ref{yto0}) and (\ref{qto0}) we know some of the terms in $\Delta^k$ converge to 0 a.s., namely
$$ - 2\langle QA(x^{k}-\hat{x}),y^{k}-y^{k-1}\rangle
+ \|y^{k}-y^{k-1}\|^{2}_{QS^{-1}} \to 0 \quad\text{a.s.} $$
hence $ \|x^k-\hat{x}\|^2_{\tau^{-1}} + \|y^k-\hat{y}\|^2_{QS^{-1}} \to \alpha $ a.s.  
Finally, the norm $\|w\|^2_R := \|x\|^2_{\tau^{-1}} + \|y\|^2_{QS^{-1}} $ is equivalent to the norm in $X\times Y$. Since the sequence $(\|w^k-\hat{w}\|_R)_{k\in\mathbb{N}}$ converges a.s., so does $(\|w^k-\hat{w}\|)_{k\in\mathbb{N}}$.
\end{proof}

\subsection{Proof of Proposition \ref{tozero1} \ref{iv}} 
The two following lemmas are consequence of (\cite{combettesPesquet}, Proposition 2.3) and their proofs are not included here. We use the standard notation $(\mathbf{\Omega},\mathcal{F},P)$ for the probability space corresponding to the random iterations $w^k$.

\begin{lemma}[\cite{combettesPesquet}, Proposition 2.3 iii)]
\label{tozero2}
Let $\mathbf{F}$ be a closed subset of a separable Hilbert space and let $(w^k)_{k\in\mathbb{N}}$ be a sequence of random variables such that the sequence $(\|w^k-w\|)_{k\in\mathbb{N}}$ converges almost surely for every $w\in\mathbf{F}$.
Then there exists $\Omega\in\mathcal{F}$ such that $\mathbb{P}(\Omega)=1$ and the sequence $(\|w^k(\omega)-w\|)_{k\in\mathbb{N}}$ converges for all  $\omega\in\Omega$ and $w\in\mathbf{F}$.
\end{lemma}

\begin{lemma}[\cite{combettesPesquet}, Proposition 2.3 iv)]
\label{tozero3}
Let $\mathbf{G}(w^k)$ be the set of cluster points of a random sequence $(w^k)_{k\in\mathbb{N}}$. Assume there exists ${\Omega}\in\mathcal{F}$ such that $\mathbb{P}({\Omega})=1$ and for every $\omega\in{\Omega}$, $\mathbf{G}(w^k(\omega))$ is nonempty and the sequence $(\|w^k(\omega)-w\|)_{k\in\mathbb{N}}$ converges for all $w\in\mathbf{G}(w^k(\omega))$.
Then $(w^k)_{k\in\mathbb{N}}$ converges almost surely to an element of $\mathbf{G}(w^k)$.
\end{lemma}

\begin{proof}[Proof of Proposition \ref{tozero1} \ref{iv}] Let $\mathbf{F}$ be the set of solutions to the saddle point problem (\ref{saddle}). 
By Proposition \ref{tozero1} iii) and Lemma~\ref{tozero2}, there exists $\Omega\in\mathcal{F}$ such that the sequence $(\|w^k(\omega)-w\|)_{k\in\mathbb{N}}$ converges for every $w\in\mathbf{F}$ and $\omega\in\Omega$. This implies, since $\mathbf{F}$ is nonempty, that $(w^k(\omega))_{k\in\mathbb{N}}$ is bounded and thus $\mathbf{G}(w^k(\omega))$ is nonempty for all $\omega\in\Omega$. By assumption, there exists $\tilde{\Omega}\in\mathcal{F}$ such that
$\textbf{G}(w^k(\omega))\subset\mathbf{F}$ for every $\omega\in\tilde{\Omega}$. Let $\omega\in\Omega\cap\tilde{\Omega}$, then  $(\|w^k(\omega)-w\|)_{k\in\mathbb{N}}$ converges for every $w\in\textbf{G}(w^k(\omega))\neq\emptyset$. By Lemma~\ref{tozero3}, we get the result.
\end{proof}

\subsection{Proof of Proposition \ref{cluster_saddle}}
The following lemma describes Algorithm~\ref{alg_spdhg} as a random sequence $(T_{j^k})_{k\in\mathbb{N}}$ of continuous operators on the primal-dual space $X\times Y$, such that the fixed points of the operators are saddle points of (\ref{saddle}):

\begin{lemma}
\label{fixpoints}
Denote $w=(w_i)_{i=0}^n=(x,y_1,...,y_n)$ and for every $j\in \{1,...,n\}$ let the operator $T_j:X\times Y \to X\times Y$ be defined by
\begin{equation*}
\begin{aligned}
\hspace{20mm}
(T_jw)_0 &= \mathrm{prox}_{\tau g}\big(x-\tau A^Ty -  \big(1+\frac{1}{p_j}\big)  \tau A_j^T((T_jw)_j-y_j)\big) \\
(T_jw)_i &= \left\{ \begin{tabular}{ll}
    $\mathrm{prox}_{\sigma_i f_i^*}(y_i+\sigma_iA_ix)$ \quad & if $i=j$ \\
    $y_i$ & else 
\end{tabular} \right.
\hspace{9mm} \text{for } 1\leq i\leq n .
\end{aligned}
\end{equation*}  
Then the iterations $w^k$ generated by Algorithm \ref{alg_spdhg} satisfy
\begin{equation} \label{itTs}
    T_{j^k}(x^{k+1},y^{k}) = (x^{k+2},y^{k+1}) .
\end{equation}
Furthermore, $\hat{w}$ is a solution to the saddle point problem (\ref{saddle}) if and only if it is a fixed point of $T_{j}$ for each $j\in \{1,...,n\}$.
\end{lemma}

\begin{proof}
By definition of the iterates in Algorithm \ref{alg_spdhg}, $(T_{j^k}(x^{k+1},y^{k}))_i = y_i^{k+1}$ for every $i\in\{1,...,n\}$. By induction it is easy to check that $z^k = A^Ty^k$, and thus
\begin{equation*}
\begin{aligned}
    \bar{z}^{k+1} &= z^k + (1+\frac{1}{p_{j^k}})\delta^k \\ 
    &= A^Ty^{k} + (1+\frac{1}{p_{j^k}}) A_{j^k}^T(y_{j^k}^{k+1}-y_{j^k}^{k}) \\
    &= A^Ty^{k} + (1+\frac{1}{p_{j^k}}) A_{j^k}^T ((T_{j^k}(x^{k+1},y^{k}))_{j^k}-y_{j^k}^{k}) .
\end{aligned}
\end{equation*}
Thus
$ (T_{j^k}(x^{k+1},y^{k}))_0 = \mathrm{prox}_{\tau g}(x^{k+1} - \tau \bar{z}^{k+1}) = x^{k+2} $
, which proves (\ref{itTs}). Now let $w$ be a fixed point of $T_j$ for every $j$. Then, for any $j$, 
$$y_j = w_j= (T_j w)_j = \text{prox}_{\sigma_i f_j^*}(y_j+\sigma_j A_jx), $$
from where it follows that, for any $j$,
\begin{equation*}
\begin{aligned}
    x = w_0 = (T_j w)_0 
    &= \text{prox}_{\tau g}(x-\tau A^Ty - (1+\frac{1}{p_j}) \tau A_j^T((T_j w)_j-y_j) \\
    &= \text{prox}_{\tau g}(x-\tau A^Ty).
\end{aligned}
\end{equation*}
These conditions on $x$ and $y$ define a saddle point (\cite{bredieslorenz}, 6.4.2). The converse result is direct.
\end{proof}

\begin{proof}[Proof of Proposition~\ref{cluster_saddle}]
Let $j^k$ be the sampling generated by the algorithm and let $z^k= (x^{k+1},y^{k})$. By Lemma \ref{fixpoints} we have $z^{k+1} = T_{j^k}z^k$ and, by~(\ref{ntozero}),
\begin{equation} \label{n-n+1}
    \mathbb{E}(\|z^{k} - z^{k-1}\|^2) =
    \mathbb{E}(\|x^{k+1} - x^k\|^2 + \|y^k-y^{k-1}\|^2)
    \to 0 .
\end{equation}
Furthermore, by the properties of the conditional expectation,
\begin{equation*}
\begin{aligned}
\mathbb{E}(\|z^{k+1} - z^k\|^2) = \mathbb{E}(\mathbb{E}^k(\|z^{k+1} - z^k\|^2)) 
&= \mathbb{E}\Big(\sum_{j=1}^n \mathbb{P}(j^k = j)\|T_j z^{k} - z^k\|^2\Big)\\
&= \sum_{j=1}^n \mathbb{P}(j^k = j)\mathbb{E}(\|T_j z^{k} - z^k\|^2) .
\end{aligned}
\end{equation*}
By assumption $p_j = \mathbb{P}(j^k=j) >0$, thus by (\ref{n-n+1}) we have
$\mathbb{E}(\|T_j z^k-z^k\|^2)\to0$ for every $j$ 
and therefore
\begin{equation} \label{Tstozero}
   T_j z^k-z^k \to 0 \quad\text{a.s. for every } j\in \{1,...,n\}.
\end{equation}
Assume now a convergent subsequence $w^{\ell_k}\to w^*$. From (\ref{n-n+1}), $ y^{k}-y^{k-1}\to0 $ a.s. and so $z^{\ell_k}$ also converges to $w^*$. By (\ref{Tstozero}) and the continuity of $T_j$, there holds
$$ w^* = \lim_{k\to\infty}z^{\ell_k} = \lim_{k\to\infty} T_{j}z^{\ell_k} = T_{j}\big(\lim_{k\to\infty}z^{\ell_k}\big) = T_jw^* \quad\text{a.s.}$$
for every $j$. Hence $w^*$ is almost surely a fixed point of $T_j$ for each $j$ and, by Lemma~\ref{fixpoints}, $w^*$ is a saddle point. 
\end{proof}

\section{Relation to Other Work}\label{sec:rel}
%
\subsection{Chambolle et al. (2018)}
In the original paper for SPDHG \cite{spdhg}, it is shown that, under Assumption~\ref{assu}, the \textit{Bregman distance} to any solution $\hat{x},\hat{y}$ of (\ref{saddle}) converges to zero, i.e. the iterates $x^k,y^k$ of Algorithm \ref{alg_spdhg} satisfy
\begin{equation} \label{gap}
    D_g^{-A^T\hat{y}}(x^k,\hat{x}) + D_{f^*}^{A\hat{x}}(y^k,\hat{y}) \to 0 \quad\text{a.s.,}
\end{equation}
where the Bregman distance is defined by $ D_h^q(u,v) = h(u) - h(v) - \langle q,u-v \rangle $ for any functional $h$ and any point $q\in\partial h(v)$ in the subdifferential of $h$. 
In \cite{spdhg}, it is shown that (\ref{gap}) implies $(x^k,y^k) \to (\hat{x},\hat{y})$ a.s.\@ if $f_i$ or $g$ are strongly convex. 


\subsection{Combettes \& Pesquet (2014)}
In \cite{combettesPesquet}, Combettes \& Pesquet look into the convergence of random sequences $(w^k)_{k\in\mathbb{N}}$ of the form 
\begin{equation} \label{T}
    w_i^{k+1} = \left\{ \begin{tabular}{ll}
        $(T w^k)_i$ & if $i\in\mathbb{S}^k$ \\
        $w_i^k$ & else
    \end{tabular} \right.
\end{equation}
where $\mathbb{S}^k\subset\{1,...,n\}$ is chosen at random. They use the Robbins-Siegmund lemma (Lemma~\ref{lemma2.2}) to prove that, for a nonexpansive operator $T$, the sequence $(w^k)_{k\in\mathbb{N}}$ converges to a fixed point of $T$. Later, Pesquet \& Repetti \cite{pesquetRepetti} used this to prove convergence for a wide class of random algorithms of the form (\ref{T}), where $T = (I+B)^{-1}$ is the resolvent operator of a monotone operator $B$.

\subsection{Alacaoglu et al. (2019)}
Recently Alacaoglu et al. also proposed a proof for the almost sure convergence of SPDHG (\cite{alacaoglu}, Theorem 4.4) using a strategy similar to ours. They describe SPDHG as a special case of a sequence $(\mathbf{w}^k)_{k\in\mathbb{N}}\subset\mathbb{R}^{dn+n^2}$ of the form (\ref{T}) for an operator $\textbf{T}:\mathbb{R}^{dn+n^2}\to\mathbb{R}^{dn+n^2}$, where $X=\mathbb{R}^d$ and $Y=\mathbb{R}^n$. They show that, under the step size condition (\ref{assu_i}), 
the iterations $\mathbf{w}^k\in\mathbb{R}^{dn+n^2}$ converge to a fixed point of $\mathbf{T}$.
In contrast, we describe SPDHG using a random sequence $(T^k)_{k\in\mathbb{N}}$ of more intuitive operators $T^k:X\times Y\to X\times Y$, and then prove the sequence $(w^k)_{k\in\mathbb{N}}$ converges to a fixed point $w\in\bigcap_{k\in\mathbb{N}}\mathrm{Fix}T^k$.


\section{Numerical Examples}
\label{sec:numeric}

In this section we use MRI reconstruction as a case study to illustrate the performance of SPDHG in comparison to PDHG. 
In parallel MRI reconstruction~\cite{mri}, a signal $x$ is reconstructed from multiple data samples $b_1,...,b_n$ of the form $b_i=A_ix+\eta_i$, where each $A_i:\mathbb{C}^d\to \mathbb{C}^m$ is an encoder operator from the signal space to the sample space, and $\eta_i\in \mathbb{C}^m$ represents noise added to the measurements.
A least-squares solution is given by 
\begin{equation} \label{mri_model}
    \hat{x} \in \arg\min_x \sum_{i=1}^n \|A_ix-b_i\|^2 + g(x)
\end{equation}
where $g$ acts as a regularizer.
We recover our convex minimization template~(\ref{min}) by identifying $\mathbb{C}^d$ with $\mathbb{R}^{2d}$ and setting $X=\mathbb{R}^{2d}$, $Y_i=\mathbb{R}^{2m}$ and $ f_i(y) = \|y - b_i\|^2 $.

Here, we consider undersampled data with $n$ coils, i.e. $A_i = S\circ F\circ C_i$, where $S:\mathbb{C}^d\to\mathbb{C}^m$ is a subsampling operator, $F:\mathbb{C}^d\to\mathbb{C}^d$ represents the discrete Fourier transform and $C_ix = c_i\cdot x $ is the element-wise multiplication of $x$ and the $i$-th coil-sensitivity map $c_i\in \mathbb{C}^d$. 
Data samples $b_i = A_ix^\dagger + \eta_i$ have been synthetically generated from an original reconstruction $x^\dagger$ from the BrainWave database \cite{MRIdata}.

In order to choose step-size parameters $\sigma_i$ and $\tau$ which comply with step-size condition (\ref{assu_i}), we consider $\sigma_i = \gamma(p_i/\|A_i\|)$ and $\tau = \gamma^{-1} (0.99/\max_i\|A_i\|)$, and then choose the value of $\gamma$ (among orders of magnitude $\gamma=10^{-5},10^{-4},...,10^{5}$) that gives the lowest objective $\Phi(x^k) = \sum_{i=1}^n f_i(Ax^k)+g(x^k)$ after 100 epochs. 
Step sizes for PDHG have been optimized in the same way by setting $\sigma = \gamma/\|A\|$  and $\tau = \gamma^{-1}(0.99/\|A\|)$, which satisfies its step-size condition $\tau\sigma\|A\|^2<1$ \cite{chambollePock}.

\begin{figure}[t]
\centering
\subfloat[\scriptsize Ground truth  ]{\includegraphics[width=0.2\textwidth]{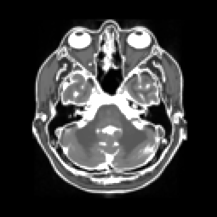}} \quad
\subfloat[\scriptsize Target  ]{\includegraphics[width=0.2\textwidth]{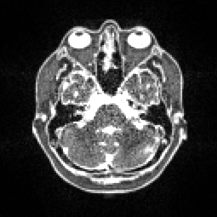}} \quad
\subfloat[\scriptsize  SPDHG 4 coils \\ \centering{ 100 epochs}]{\includegraphics[width=0.2\textwidth]{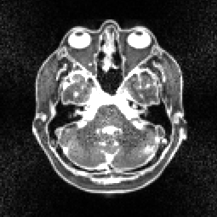}} \quad
\subfloat[\scriptsize  PDHG 4 coils \\ \centering{ 100 iterations}]{\includegraphics[width=0.2\textwidth]{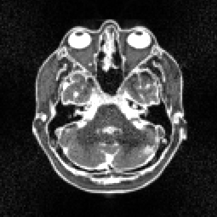}} \\
\vspace{-3mm}
\subfloat[\scriptsize  SPDHG 8 coils \\ \centering{ 100 epochs}]{\includegraphics[width=0.2\textwidth]{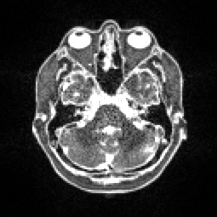}} \quad
\subfloat[\scriptsize  PDHG 8 coils \\ \centering{ 100 iterations}]{\includegraphics[width=0.2\textwidth]{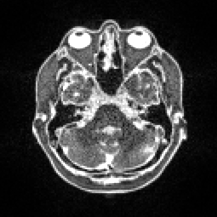}} \quad
\subfloat[\scriptsize  SPDHG 8 coils \\ \centering{ 1,000 epochs}]{\includegraphics[width=0.2\textwidth]{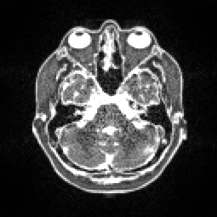}} \quad
\subfloat[\scriptsize  PDHG 8 coils \\ \centering{ 1,000 iterations}]{\includegraphics[width=0.2\textwidth]{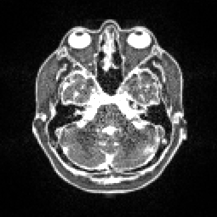}}
\caption{Images reconstructed using different algorithms. Subfigure (a) shows the original ground truth $x^\dagger$ from which the noisy data samples have been synthetically generated, while subfigure (b) shows a reliable reconstruction $x^*$ to which the other figures should be compared. }
\label{fig: reconL2}
\end{figure}

Figure \ref{fig: reconL2} shows reconstructions obtained through SPDHG and PDHG. 
Here, we considered model (\ref{mri_model}) with squared 2-norm regularizer $g(x) = 10^{-4}\|x\|^2$. The \emph{target} solution $x^*$ in (b) has been computed by running SPDHG for large number of epochs ($>10^4$). Comparing subfigures (e) and (f), the solution for SPDHG seems to be closer to the target (b) than that of PDHG at 100 epochs. At 1,000 epochs, solutions (g) and (h) appear visually similar to each other.

Figure \ref{fig: reconTV} shows reconstructions using model (\ref{mri_model}) with total-variation regularizer $g(x)=10^{-4}\|\nabla x\|_1$. As before, solutions by SPDHG appear closer to the target than those by PDHG. Notice as well how increasing the number of coils seems to yield more detailed although noisier reconstructions. 

Figure \ref{fig: convergence plot} summarizes the performance of both algorithms for all of these reconstructions and some more.

\section{Conclusions}

In this paper we showed the almost sure convergence of SPDHG by using alternative arguments as provided by Alacaoglu et al \cite{alacaoglu}. We also investigated SPDHG in the context of parallel MRI where we observed a significant speed-up, thereby supporting the practical use of variational regularization methods for this application.


\begin{figure}[ht]
\centering
\subfloat[\scriptsize Ground truth ]{\includegraphics[width=0.2\textwidth]{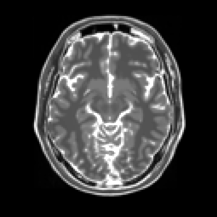}} \quad
\subfloat[\scriptsize Target ]{\includegraphics[width=0.2\textwidth]{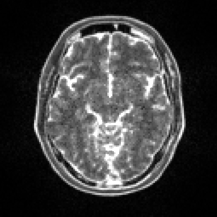}} \quad
\subfloat[\scriptsize  SPDHG 4 coils \\ \centering{ 100 epochs}]{\includegraphics[width=0.2\textwidth]{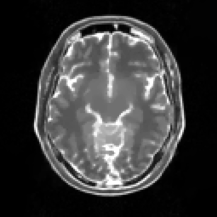}} \quad
\subfloat[\scriptsize  PDHG 4 coils \\ \centering{ 100 iterations}]{\includegraphics[width=0.2\textwidth]{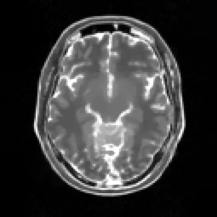}} \\
\vspace{-3mm}
\subfloat[\scriptsize  SPDHG 8 coils \\ \centering{ 100 epochs}]{\includegraphics[width=0.2\textwidth]{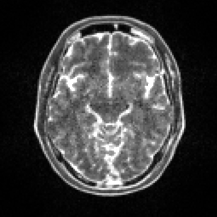}} \quad
\subfloat[\scriptsize  PDHG 8 coils \\ \centering{ 100 iterations}]{\includegraphics[width=0.2\textwidth]{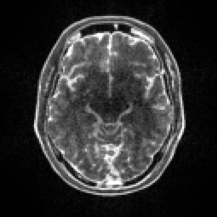}} \quad
\subfloat[\scriptsize  SPDHG 8 coils \\ \centering{ 1,000 epochs}]{\includegraphics[width=0.2\textwidth]{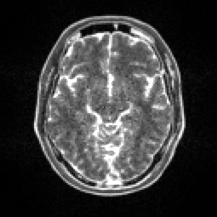}} \quad
\subfloat[\scriptsize  PDHG 8 coils \\ \centering{ 1,000 iterations}]{\includegraphics[width=0.2\textwidth]{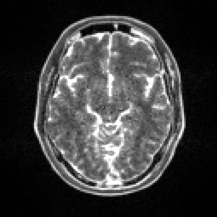}}
\caption{Images reconstructed using noisy data samples synthetically generated from the ground truth~(a). Subfigure (b) shows the target reconstruction for comparison. }
\label{fig: reconTV}
\end{figure}

\begin{figure}[ht]
\centering
\subfloat[\scriptsize  $L^2$-regularizer]{\includegraphics[width=0.53\textwidth]{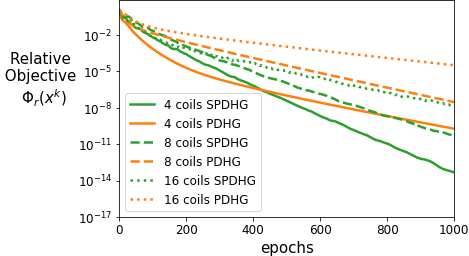}} \quad
\subfloat[\scriptsize  $TV$-regularizer]{\includegraphics[width=0.44\textwidth]{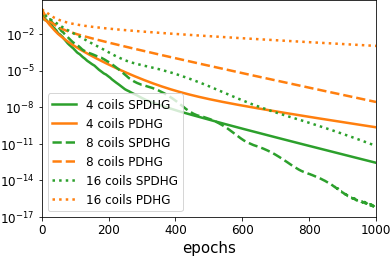}} 
\caption{ Relative objective $\Phi_r(x^k) = \frac{\Phi(x^k) - \Phi(x^*)}{\Phi(x^0) - \Phi(x^*)}$ at every epoch for image reconstruction, including the examples from Figures \ref{fig: reconL2} and \ref{fig: reconTV}.}
\label{fig: convergence plot}
\end{figure}


\bibliographystyle{splncs04}
\bibliography{references}

\begin{thebibliography}{10}
\providecommand{\url}[1]{\texttt{#1}}
\providecommand{\urlprefix}{URL }
\providecommand{\doi}[1]{https://doi.org/#1}

\bibitem{alacaoglu}
Alacaoglu, A., Fercoq, O., Cevher, V.: On the convergence of stochastic
  primal-dual hybrid gradient. arXiv preprint arXiv:1911.00799  (2019)

\bibitem{boserGuyonVapnik}
Boser, B.E., Guyon, I.M., Vapnik, V.N.: A training algorithm for optimal margin
  classifiers. In: Proceedings of the fifth annual workshop on Computational
  learning theory. pp. 144--152 (1992)

\bibitem{bredieslorenz}
Bredies, K., Lorenz, D.: Mathematical Image Processing. Springer (2018)

\bibitem{cevherBigData}
Cevher, V., Becker, S., Schmidt, M.: Convex optimization for big data:
  Scalable, randomized, and parallel algorithms for big data analytics. IEEE
  Signal Processing Magazine  \textbf{31}(5),  32--43 (2014)

\bibitem{spdhg}
Chambolle, A., Ehrhardt, M.J., Richt{\'a}rik, P., Schönlieb, C.B.: Stochastic
  primal-dual hybrid gradient algorithm with arbitrary sampling and imaging
  applications. SIAM Journal on Optimization  \textbf{28}(4),  2783--2808
  (2018)

\bibitem{chambollePock}
Chambolle, A., Pock, T.: A first-order primal-dual algorithm for convex
  problems with applications to imaging. Journal of mathematical imaging and
  vision  \textbf{40}(1),  120--145 (2011)

\bibitem{chambolle2016introduction}
Chambolle, A., Pock, T.: An introduction to continuous optimization for
  imaging. Acta Numerica  \textbf{25},  161--319 (2016)

\bibitem{MRIdata}
Cocosco, C., Kollokian, V., Kwan, R.S., Evans, A.: {BrainWeb: Online Interface
  to a 3D MRI Simulated Brain Database}. NeuroImage  \textbf{5}(4), ~425 (1997)

\bibitem{combettesPesquet}
Combettes, P.L., Pesquet, J.C.: {Stochastic quasi-Fej{\'e}r block-coordinate
  fixed point iterations}. SIAM Journal on Optimization  \textbf{25}(2),
  1221--1248 (2015)

\bibitem{ehrhardt2019faster}
Ehrhardt, M.J., Markiewicz, P., Sch{\"o}nlieb, C.B.: Faster {PET}
  reconstruction with non-smooth priors by randomization and preconditioning.
  Physics in Medicine \& Biology  \textbf{64}(22),  225019 (2019)

\bibitem{esserZhangChan}
Esser, E., Zhang, X., Chan, T.F.: A general framework for a class of first
  order primal-dual algorithms for convex optimization in imaging science. SIAM
  Journal on Imaging Sciences  \textbf{3}(4),  1015--1046 (2010)

\bibitem{fercoq_etal2019}
Fercoq, O., Alacaoglu, A., Necoara, I., Cevher, V.: Almost surely constrained
  convex optimization. arXiv preprint arXiv:1902.00126  (2019)

\bibitem{fercoqBianchi2019}
Fercoq, O., Bianchi, P.: A coordinate-descent primal-dual algorithm with large
  step size and possibly nonseparable functions. SIAM Journal on Optimization
  \textbf{29}(1),  100--134 (2019)

\bibitem{mri}
{Fessler}, J.A.: Optimization methods for magnetic resonance image
  reconstruction: Key models and optimization algorithms. IEEE Signal
  Processing Magazine  \textbf{37}(1),  33--40 (2020)

\bibitem{gao2019randomPD}
Gao, X., Xu, Y.Y., Zhang, S.Z.: Randomized primal--dual proximal block
  coordinate updates. Journal of the Operations Research Society of China
  \textbf{7}(2),  205--250 (2019)

\bibitem{latafat2019randomPD}
Latafat, P., Freris, N.M., Patrinos, P.: A new randomized block-coordinate
  primal-dual proximal algorithm for distributed optimization. IEEE
  Transactions on Automatic Control  \textbf{64}(10),  4050--4065 (2019)

\bibitem{patrascuNecoara2017}
Patrascu, A., Necoara, I.: Nonasymptotic convergence of stochastic proximal
  point methods for constrained convex optimization. The Journal of Machine
  Learning Research  \textbf{18}(1),  7204--7245 (2017)

\bibitem{pesquetRepetti}
Pesquet, J.C., Repetti, A.: A class of randomized primal-dual algorithms for
  distributed optimization. arXiv preprint arXiv:1406.6404  (2014)

\bibitem{CremersChambollePock}
{Pock}, T., {Cremers}, D., {Bischof}, H., {Chambolle}, A.: A algorithm for
  minimizing the {Mumford-Shah} functional. In: 2009 IEEE 12th International
  Conference on Computer Vision. pp. 1133--1140 (2009)

\bibitem{robbins1971}
Robbins, H., Siegmund, D.: A convergence theorem for non negative almost
  supermartingales and some applications. In: Optimizing methods in statistics,
  pp. 233--257. Elsevier (1971)

\bibitem{ROF}
Rudin, L.I., Osher, S., Fatemi, E.: Nonlinear total variation based noise
  removal algorithms. Physica D: nonlinear phenomena  \textbf{60}(1-4),
  259--268 (1992)

\bibitem{shalevZhang}
Shalev-Shwartz, S., Zhang, T.: Stochastic dual coordinate ascent methods for
  regularized loss minimization. Journal of Machine Learning Research
  \textbf{14}(Feb),  567--599 (2013)

\bibitem{zhangXiao}
Zhang, Y., Xiao, L.: Stochastic primal-dual coordinate method for regularized
  empirical risk minimization. The Journal of Machine Learning Research
  \textbf{18}(1),  2939--2980 (2017)

\end{thebibliography}
\end{document}